\documentclass[12pt]{amsart}
\usepackage{colortbl}
\usepackage{longtable}
\usepackage{hyperref}
\usepackage{amsmath}
\usepackage{pslatex}
\usepackage{amsthm}
\usepackage{amssymb}
\usepackage{latexsym}
\usepackage{graphicx}
\usepackage{multicol}
\usepackage{tikz}
\usepackage{pst-plot}
\usepackage{pstricks}

\hypersetup{hypertex, colorlinks=true, linkcolor=blue, filecolor=blue, pagecolor=blue, urlcolor=blue}

\newtheorem{theorem}{Theorem}
\newtheorem{lemma}{Lemma}
\newtheorem{corollary}{Corollary}

\begin{document}

\author{Skyler Simmons}
\address{Department of Mathematics, Utah Valley University, Orem, UT 84058}
\email{skyler.simmons@uvu.edu}

\title{The Eight-Body Cubic Collision-Based Periodic Orbit}

\begin{abstract}
We construct a highly-symmetric periodic orbit of eight bodies in three dimensions.  In this orbit, each body collides with its three nearest neighbors in a regular periodic fashion.  Regularization of the collisions in the orbit is achieved by an extension of the Levi-Civita method.  Initial conditions for the orbit are found numerically.  Linear stability of the orbit is then shown using a technique by Roberts.  Evidence toward higher-order stability is presented as an additional result of a numerical calculation.
\end{abstract}


\maketitle

\section{Introduction}
In the \textit{Principia Mathematica} (see \cite{bibNewton}), Newton gives mathematical equations governing the motion of point masses within their mutual gravitational field.  Specifically, for $n$ point masses in $\mathbb{R}^d$ located at $\mathbf{x}_i$ with mass $m_i$ for $i = 1, 2, ..., n$, we have that
\begin{equation}
\label{NewtonEquation}
m_i\ddot{\mathbf{x}}_i = \sum_{i \neq j} \frac{Gm_im_j}{|\mathbf{x}_i - \mathbf{x}_j|^2} \left(\frac{\mathbf{x}_i - \mathbf{x}_j}{|\mathbf{x}_i - \mathbf{x}_j|}\right).
\end{equation}
Here, the dot represents the derivative with respect to time, and $G$ is a constant.  In SI units, $G = 6.67430 \times 10^{-11} \text{m}^3 / \text{kg}^2\text{s}$)  A suitable choice of units gives $G = 1$, which is often assumed for mathematical simplicity. \\

\textit{Collision singularities} of the $n$-body problem occur when $\mathbf{x}_i = \mathbf{x}_j$ for some $i \neq j$.  Under suitable conditions, collisions of two bodies can be regularized.  \textit{Regularization} involves a change of temporal and spatial variables so that the collision point becomes a regular point for the differential equations.  Collision singularities have received a great deal of study.  Of particular note is a result by McGehee \cite{bibMcGehee1}, which shows that in general, a collision of three or more bodies cannot be regularized. \\

Many periodic orbits featuring collisions have been produced.  Existence, stability, and other properties of periodic orbits with three bodies in one spatial dimension are studied in both analytical and numerical contexts as early as 1956 in \cite{bibJS} and as recently as 2019 in \cite{bibKTXY}.  Works between these years include \cite{bibHE}, \cite{bibHM}, \cite{bibST1}, \cite{bibMoeckel}, \cite{bibVE}, \cite{bibST2}, \cite{bibST3}, \cite{bibShib}, \cite{bibYan3}, and \cite{bibOY}.  Orbits with four bodies in one spatial dimension are featured in \cite{bibShib}, \cite{bibMartinez}, \cite{bibHuang} and \cite{bibYan1}.  Orbits in two spatial dimensions featuring collisions were studied as early as 1979 in \cite{bibBroucke} and as recently as 2021 in \cite{bibSim}, with other notable works including \cite{bibRoySteves}, \cite{bibBOYSR}, \cite{bibSSS}, \cite{bibBOYS}, \cite{bibBMS}, \cite{bibOY3}, \cite{bibYan2}, and \cite{bibBS1}.  Additionally, in \cite{bibShib} and \cite{bibMartinez}, large families of highly-symmetric orbits are given in one, two and three dimensions, all of which can be expressed in two degrees of freedom. Additionally, three-dimensional restricted collision-based orbits are studied in \cite{bibMISC}, \cite{bibBV}, \cite{bibBDV}, and \cite{bibGPSV} as a case of the $e = 1$ Sitnikov problem, which can be reduced to a time-dependent two-degree-of-freedom problem.\\

This paper studies a three-degree of freedom, highly-symmetric, periodic orbit of eight bodies featuring collisions.  The bodies form the vertices of a rectangular prism at all points in time, with edges parallel to the standard coordinate axes in $\mathbb{R}^3$.  Each body collides with its three nearest neighbors in a regular periodic fashion.  This appears to be the first three-degree-of-freedom collision-based periodic orbit studied. \\

The remainder of the paper is as follows: In Section \ref{secSetting}, we set up and regularize the Hamiltonian that corresponds to the configuration being considered.  Section \ref{secPeriodicOrbit} details the construction of the periodic orbit.  We first  describe the orbit in the regularized setting.  Then, we analytically establish sufficient conditions for the orbit to exist.  Finally, we complete the existence proof with a numerical calculation. \\

Section \ref{secStabSim} establishes the linear stability of the orbit.  We first review some preliminary details of stability, including linear stability.  Next, we establish notation for the symmetries of the orbit.  We next detail some results by Roberts in \cite{bibRoberts8} that allow us to establish the linear stability of the orbit in a rigorous numerical fashion in terms of these symmetries.  Applications to the orbit under consideration are detailed after each result.  Finally, in Section \ref{secStabRes}, we give results of the numerical stability calculation established in the previous section, as well as some further numerical evidence of higher-order stability of the orbit. \\

\section{The Hamiltonian Setting and Regularization}
\label{secSetting}

\subsection{Configuration}
\label{subConfig}
We consider the Newtonian 8-body problem with point unit masses located at $(\pm q_1, \pm q_2, \pm q_3)$, where the choices of sign are taken independently of each other.  Although the symmetries of the position give the vertices of a rectangular prism, we will refer to this as the \textit{cubic configuration} (see Figure \ref{CubicConfig}). \\

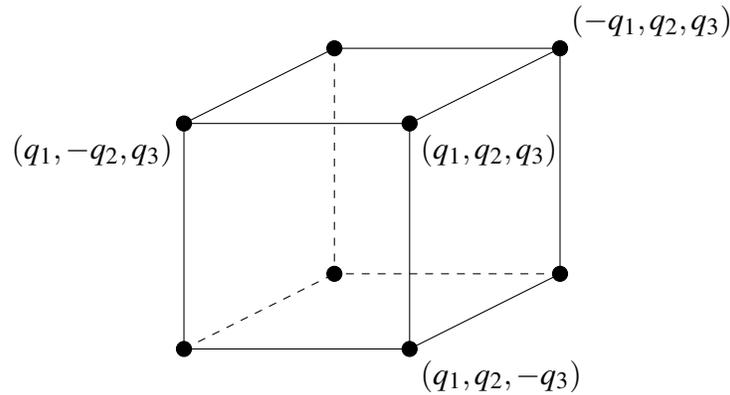
\begin{figure}[h]
\begin{tikzpicture}[scale = 1]
\draw[-] (0,0) -- (3,0) -- (3,3) -- (0,3) -- (0,0);
\draw[-] (5,1) -- (5,4) -- (2,4);
\draw[dashed] (0,0) -- (2,1);
\draw[dashed] (5,1) -- (2,1) -- (2,4);
\draw[-] (3,0) -- (5,1);
\draw[-] (3,3) -- (5,4);
\draw[-] (0,3) -- (2,4);
\draw[fill] (0,0) circle [radius = 0.1];
\draw[fill] (3,0) circle [radius = 0.1] node [below right] {$(q_1, q_2, -q_3)$}; 
\draw[fill] (0,3) circle [radius = 0.1] node [below left] {$(q_1, -q_2, q_3)$};
\draw[fill] (3,3) circle [radius = 0.1] node [below right] {$(q_1, q_2, q_3)$};
\draw[fill] (2,1) circle [radius = 0.1];
\draw[fill] (5,1) circle [radius = 0.1];
\draw[fill] (2,4) circle [radius = 0.1];
\draw[fill] (5,4) circle [radius = 0.1] node[above right] {$(-q_1, q_2, q_3)$};
\end{tikzpicture}
\caption{The cubic configuration.}
\label{CubicConfig}
\end{figure} 

Note that when $q_1 = 0$, if $q_2q_3 \neq 0$, then we have four pairs of bodies colliding in the $x = 0$ plane.  Similar results hold in the $y = 0$ and $z = 0$ planes by permuting the subscripts.  We seek an orbit possessing these four-pair collisions in the $x = 0$, $y = 0$, and $z = 0$ planes in a periodic fashion as pictured in Figure \ref{PeriodicOrbitFigure}. \\

\begin{figure}[h]
\includegraphics[scale=.5]{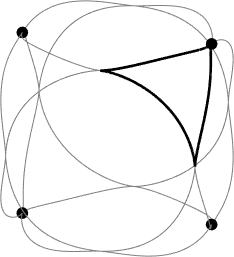} \quad \quad
\includegraphics[scale=.5]{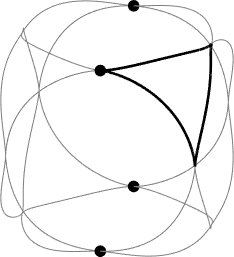} \quad \quad
\includegraphics[scale=.5]{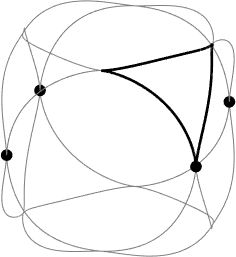}
\caption{The eight-body cubic collision orbit.  Four simultaneous binary collisions occur in the $x = 0$, $y = 0$, and $z = 0$ planes in turn as pictured.  For clarity, the trajectory of one of the eight bodies is highlighted.}
\label{PeriodicOrbitFigure}
\end{figure}

\subsection{The Hamiltonian Setting}
\label{subHamiltonian}
The potential energy is the sum of 28 terms.  For convenience, these are divided up into \textit{cube diagonals}, \textit{face diagonals}, and \textit{edges}. \\

\begin{figure}[h]
\begin{tikzpicture}[scale = 1]
\draw[-] [gray] (0,0) -- (3,0) -- (3,3) -- (0,3) -- (0,0);
\draw[-] [gray] (5,1) -- (5,4) -- (2,4);
\draw[dashed] [gray] (0,0) -- (2,1);
\draw[dashed] [gray] (5,1) -- (2,1) -- (2,4);
\draw[-] [gray] (3,0) -- (5,1);
\draw[-] [gray] (3,3) -- (5,4);
\draw[-] [gray] (0,3) -- (2,4);
\draw[-] (0,0) -- (5,4);
\draw[-] (3,0) -- (2,4);
\draw[-] (0,3) -- (5,1);
\draw[-] (3,3) -- (2,1);
\draw[fill] [gray] (0,0) circle [radius = 0.1];
\draw[fill] [gray] (3,0) circle [radius = 0.1]; 
\draw[fill] [gray] (0,3) circle [radius = 0.1];
\draw[fill] [gray] (3,3) circle [radius = 0.1];
\draw[fill] [gray] (2,1) circle [radius = 0.1];
\draw[fill] [gray] (5,1) circle [radius = 0.1];
\draw[fill] [gray] (2,4) circle [radius = 0.1];
\draw[fill] [gray] (5,4) circle [radius = 0.1];
\end{tikzpicture}
\caption{Cube diagonals (4 total).}
\label{CubeDiagonals}
\end{figure}
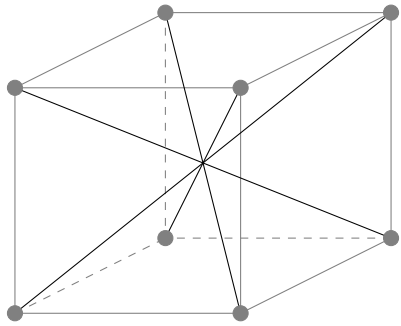 

\begin{figure}[h]
\begin{tikzpicture}[scale = 1]
\draw[-] [gray] (0,0) -- (3,0) -- (3,3) -- (0,3) -- (0,0);
\draw[-] [gray] (5,1) -- (5,4) -- (2,4);
\draw[dashed] [gray] (0,0) -- (2,1);
\draw[dashed] [gray] (5,1) -- (2,1) -- (2,4);
\draw[-] [gray] (3,0) -- (5,1);
\draw[-] [gray] (3,3) -- (5,4);
\draw[-] [gray] (0,3) -- (2,4);
\draw[-] (0,0) -- (3,3);
\draw[-] (0,3) -- (3,0);
\draw[-] (3,0) -- (5,4);
\draw[-] (3,3) -- (5,1);
\draw[-] (0,3) -- (5,4);
\draw[-] (3,3) -- (2,4);
\draw[fill] [gray] (0,0) circle [radius = 0.1];
\draw[fill] [gray] (3,0) circle [radius = 0.1]; 
\draw[fill] [gray] (0,3) circle [radius = 0.1];
\draw[fill] [gray] (3,3) circle [radius = 0.1];
\draw[fill] [gray] (2,1) circle [radius = 0.1];
\draw[fill] [gray] (5,1) circle [radius = 0.1];
\draw[fill] [gray] (2,4) circle [radius = 0.1];
\draw[fill] [gray] (5,4) circle [radius = 0.1];
\end{tikzpicture}
\caption{Face diagonals (12 total -- the remaining six are on the opposite faces of the cube).}
\label{FaceDiagonals}
\end{figure}
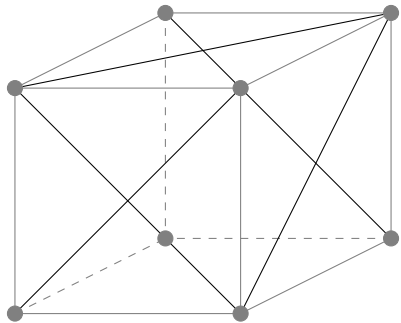 

Each of the four cube diagonals (see Figure \ref{CubeDiagonals}) contributes a term of the form
$$\frac{1}{\sqrt{(2q_1)^2 + (2q_2)^2 + (2q_3)^2}} = \frac{1}{2\sqrt{q_1^2 + q_2^2 + q_3^2}}.$$
Let $\mathcal{I} = \{1,2,3\}$.  Each face diagonal (see Figure \ref{FaceDiagonals}) contributes a term of the form
$$\frac{1}{\sqrt{(2q_i)^2 + (2q_j)^2}} = \frac{1}{2\sqrt{q_i^2 + q_j^2}},$$
with $i, j \in \mathcal{I}$ and $i \neq j$.  Specifically are four terms for each of the three possible choices of indices.
Lastly, each edge contributes a term of the form
$$\frac{1}{\sqrt{(2q_i)^2}} = \frac{1}{2q_i},$$
with $i \in \mathcal{I}$.  Again, for each index there are four terms.  Hence, the total potential energy of the system is
$$U = \frac{2}{\sqrt{q_1^2 + q_2^2 + q_3^2}} + \frac{2}{\sqrt{q_1^2 + q_2^2}} + \frac{2}{\sqrt{q_1^2 + q_3^2}} + \frac{2}{\sqrt{q_2^2 + q_3^2}} + \frac{2}{q_1} + \frac{2}{q_2} + \frac{2}{q_3}.$$

Let $p_i = \dot{q}_i$ denote the components of the momentum of the bodies.  The kinetic energy for the system is 
$$K = \frac{8\left(\sqrt{p_1^2 + p_2^2 + p_3^2}\right)^2}{2} = 4\left(p_1^2 + p_2^2 + p_3^2\right).$$
The Hamiltonian for the system is then given by $H = K - U$. \\

\subsection{Regularization}
\label{subReg}
We regularize the collisions that occur at $q_i = 0$ using an extension of the Levi-Civita method (see \cite{bibLeviCivita}).  Specifically, let
$$F = \sum_{i \in \mathcal{I}} \sqrt{q_i}P_i.$$
This generates a coordinate transformation given by
$$Q_i = \frac{\partial F}{\partial P_i} = \sqrt{q_i} \quad p_i = \frac{\partial F}{\partial q_i} = \frac{P_i}{2\sqrt{q_i}},$$
or
$$q_i = Q_i^2 \quad p_i = \frac{P_i}{2Q_i}.$$

In these coordinates, the potential energy for the system is given by
$$\tilde{U} = \frac{2}{\sqrt{Q_1^4 + Q_2^4 + Q_3^4}} + \frac{2}{\sqrt{Q_1^4 + Q_2^4}} + \frac{2}{\sqrt{Q_1^4 + Q_3^4}} + \frac{2}{\sqrt{Q_2^4 + Q_3^4}} + \frac{2}{Q_1^2} + \frac{2}{Q_2^2} + \frac{2}{Q_3^2}.$$
The new kinetic energy is given by
$$\tilde{K} = \frac{P_1^2}{Q_1^2} + \frac{P_2^2}{Q_2^2} + \frac{P_3^2}{Q_3^2}.$$
The new Hamiltonian is given by $\tilde{H} = \tilde{K} - \tilde{U}.$ \\

Lastly, to regularize the collisions at $Q_i = 0$, we apply a change of time satisfying
$$\frac{dt}{ds} = Q_1^2Q_2^2Q_3^2.$$
This gives the regularized Hamiltonian $\Gamma = \frac{dt}{ds}(\tilde{H} - E),$ or
\begin{align*}
\Gamma &= P_1^2Q_2^2Q_3^2 + Q_1^2P_2^2Q_3^2 + Q_1^2Q_2^2P_3^2 \\ 
&- \frac{2Q_1^2Q_2^2Q_3^2}{\sqrt{Q_1^4 + Q_2^4 + Q_3^4}} - \frac{2Q_1^2Q_2^2Q_3^2}{\sqrt{Q_1^4 + Q_2^4}} - \frac{2Q_1^2Q_2^2Q_3^2}{\sqrt{Q_1^4 + Q_3^4}} - \frac{2Q_1^2Q_2^2Q_3^2}{\sqrt{Q_2^4 + Q_3^4}} \\
&- 2Q_2^2Q_3^2 - 2Q_1^2Q_3^2 - 2Q_1^2Q_2^2 - EQ_1^2Q_2^2Q_3^2,
\end{align*}
where $E$ is the fixed energy of the system. \\

We now show that the system has been regularized as claimed.  Let $i, j, k \in \mathcal{I}$ be distinct.  Then, at the collision where $Q_i = 0$, $Q_j \neq 0$, and $Q_k \neq 0$, the condition $\Gamma = 0$ forces
\begin{equation}
\label{regularized}
P_i^2Q_j^2Q_k^2 - 2Q_j^2Q_k^2 = (P_i^2 - 2)Q_j^2Q_k^2 = 0.
\end{equation}
Then $P_i = \pm \sqrt{2}$.  Moreover, since 
\begin{equation}
\label{collisionCondition}
\dot{Q}_i = \frac{d\Gamma}{dP_i} = 2P_iQ_j^2Q_k^2
\end{equation}
then $\dot{Q}_i \neq 0$ when $Q_i = 0$.  Hence the orbit can be continued past the collision.  \\

An important feature of the regularization that can be determined from Equation \ref{collisionCondition} is that both $\dot{Q}_i$ and $P_i$ have the same sign at the collision time.  Since $P_i$ is continuous, then in an open time interval containing the collision the sign of $P_i$ does not change.  Hence, the sign of $\dot{Q}_i$ also does not change, so $Q_i$ must either pass from a negative to a positive value at collision, or from a positive to a negative one.

\section{The Periodic Orbit}
\label{secPeriodicOrbit}
\subsection{Description}
\label{subDescription}
The desired orbit passes through four simultaneous binary collisions in the $x = 0$, $y = 0$, and $z = 0$ planes in a periodic fashion, as pictured in Figure \ref{PeriodicOrbitFigure}.  In a physical sense, we start with the bodies with (non-regularized) positions given by
$$(\pm q_1, \pm q_2, \pm q_3) = (0, \omega, \omega)$$
and ending at
$$(\pm q_1, \pm q_2, \pm q_3) = (\omega, 0, \omega),$$
for some positive number $\omega$.  The proposed orbit will then be extended by a symmetry coinciding with a rotation of $120^\circ$ about the line $x = y = z$ in $\mathbb{R}^3$.  In other words, the orbit continues through a sequence of collisions 
$$(\pm q_1, \pm q_2, \pm q_3):(0, \omega, \omega) \to (\omega, 0, \omega) \to (\omega, \omega, 0) \to (0, \omega, \omega) \to \ldots,$$
with the collisions being equally-spaced in time. \\

In the regularized coordinates, the velocity components can also be defined.  Let $\gamma(s) = (Q_1(s), Q_2(s), Q_3(s), P_1(s), P_2(s), P_3(s))^T$.  At each collision time with $Q_i = 0$, the sign of $Q_i$ changes as noted at the end of Section \ref{subReg}.  Additionally, both $\gamma(s)$ and $-\gamma(s)$ correspond to the same setting in the original coordinates.  Hence, in the regularized setting, one period of the orbit passes through six collisions rather than three.  \\

\subsection{Extension by Symmetry}
\label{subExtension}

\begin{lemma}
\label{symmetryconstruct}
Suppose $\gamma(s)$ is a solution to the regularized Hamiltonian system $\Gamma$ that satisfies
\begin{align*}
\gamma(0) &= (0, \alpha, \alpha, \sqrt{2}, -\beta, \beta)^T \\
\gamma(2\tau) &= (\alpha, 0, \alpha, \beta, -\sqrt{2}, -\beta)^T.
\end{align*}
for some $\tau > 0$ and $E < 0$.  Then $\gamma(s)$ extends to a $12\tau$ periodic orbit for the system $\Gamma$.
\end{lemma}

\begin{proof}
We first establish the symmetries that will allow us to extend the orbit as claimed.  By direct calculation, we find that the equation for $\dot{Q}_i$ is negated under the transformation $P_i \mapsto -P_i$ and remains fixed under any sign change of the remaining variables.  We also find that $\dot{P}_i$ is negated under $Q_i \mapsto -Q_i$ and remains fixed under any other sign change of the remaining variables.  Furthermore, for any permutation $\sigma \in S_3$, since $\Gamma$ is fixed under permutation of the subscripts by $\sigma$, then the equation of motion for $\dot{Q}_i$ in terms of $Q_1, Q_2, Q_3, P_1, P_2, P_3$ is the same as that of $\dot{Q}_{\sigma(i)}$ in terms of $Q_{\sigma(1)}, Q_{\sigma(2)}, Q_{\sigma(3)}, P_{\sigma(1)}, P_{\sigma(2)}, P_{\sigma(3)}$.  Similar permutation results hold for $P_i$. \\

Consider the orbit with initial conditions $\gamma(2\tau)$.  By the symmetries just discussed, we have that
\begin{align*}
\dot{Q}_1(2\tau) &= \dot{Q}_3(0) \\
\dot{Q}_2(2\tau) &= -\dot{Q}_1(0) \\
\dot{Q}_3(2\tau) &= \dot{Q}_2(0) \\
\dot{P}_1(2\tau) &= \dot{P}_3(0) \\
\dot{P}_2(2\tau) &= -\dot{P}_1(0) \\
\dot{P}_3(2\tau) &= \dot{P}_2(0).
\end{align*}
Moreover, the equations of motion are the same as those on the interval $s \in [0, 2\tau]$ under permutations and sign changes as discussed above.  Existence and uniqueness of solutions to differential equations gives
\begin{align*}
Q_1(s + 2\tau) &= Q_3(s) \\
Q_2(s + 2\tau) &= -Q_1(s) \\
Q_3(s + 2\tau) &= Q_2(s) \\
P_1(s + 2\tau) &= P_3(s) \\
P_2(s + 2\tau) &= -P_1(s) \\
P_3(s + 2\tau) &= P_2(s)
\end{align*}
is a solution to the Hamiltonian system given by $\Gamma$.  Setting $s = 2\tau$, we have that
$$\gamma(4\tau) = (\alpha, -\alpha, 0, -\beta, -\beta, -\sqrt{2})^T.$$
Repeating the argument with initial conditions given by $\gamma(4\tau)$ gives
$$\gamma(6\tau) = (0, -\alpha, -\alpha, -\sqrt{2}, \beta, -\beta)^T.$$
Continuing in turn, we have that
\begin{align*}
\gamma(8\tau) &= (-\alpha, 0, -\alpha, -\beta, \sqrt{2}, \beta)^T \\
\gamma(10\tau) &=(-\alpha, \alpha, 0, \beta, \beta, \sqrt{2})^T \\
\gamma(12\tau) &=(0, \alpha, \alpha, \sqrt{2}, -\beta, \beta)^T.
\end{align*}
Since $\gamma(0) = \gamma(12\tau)$, the periodic orbit has been constructed as claimed.
\end{proof}

Physically, the orbit constructed in Lemma \ref{symmetryconstruct} corresponds to an orbit in which all bodies start in the $x = 0$ plane at collisions symmetrically placed along the lines $y = \pm z$.  The velocity of each body projected onto the $x = 0$ plane is orthogonal to the projection of its position.  The orbit then proceeds to collisions in the $y = 0$ and $z = 0$ planes with similarly symmetric positions and velocities. \\

\textbf{Note:} We do not rule out the possibility of the existence of a ``less symmetric'' orbit.  Indeed, the arguments in Lemma \ref{symmetryconstruct} give the same conclusion if we assume that
\begin{align*}
\gamma(0) &= (0, a, b, \sqrt{2}, -c, d)^T \\
\gamma(2\tau) &= (b, 0, a, d, -\sqrt{2}, -c)^T
\end{align*}
without the requirement that $a = b$ and $c = d$.  However, for simplicity we restrict ourselves to the ``reduced'' case at the present time.\\

\subsection{Numerical Results}
\label{subNumerICond}

Initial conditions for the orbit are found numerically using a shooting method.  Specifically, using initial conditions of the form $\gamma(0)$ as in Lemma \ref{symmetryconstruct} and letting $6\tau >0$ be the first time when $Q_1(s) = 0$, we find values of $\alpha$ and $\beta$ that minimize the value of $||\gamma(0) + \gamma(6\tau)||$, as we expect that $\gamma(6\tau) = -\gamma(0)$ from Lemma \ref{symmetryconstruct}.  Hence, a simple numerical minimization is an appropriate approach.  As $||\gamma(0) + \gamma(6\tau)||$ may have many local minima, we then verify that for these values of $\alpha$ and $\beta$ we have the conditions at $\gamma(2\tau)$ as specified in Lemma \ref{symmetryconstruct}.  The appropriate values satisfying these conditions with $E = -1$ are 
\begin{equation}
\alpha = 3.100685, \quad \beta = 0.668162.
\label{alphabeta}
\end{equation}
The full period of the regularized orbit is given by
\begin{equation}
12\tau = 0.124736
\label{periodlength}
\end{equation}

\begin{figure}[h]
\includegraphics[scale=.65]{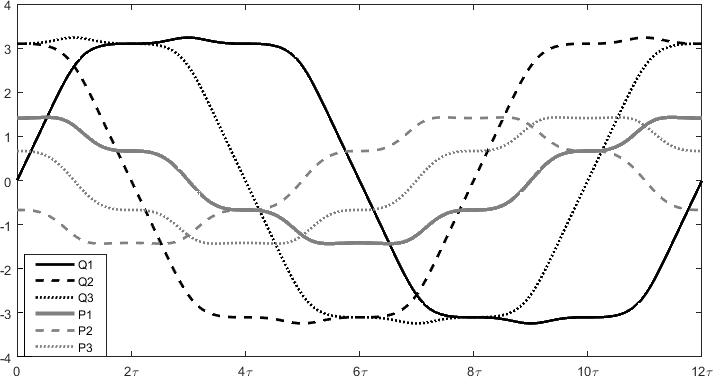}
\caption{Integration of the regularized equations of motion with the initial conditions in Section \ref{subNumerICond}.}
\label{RegularizedIntegration}
\end{figure}

\section{Stability and Symmetry}
\label{secStabSim}

\subsection{Definitions and Preliminaries}
\label{subStabDef}
Let $\mathcal{O}(\gamma_0)$ be the set of all points in $\mathbb{R}^6$ traced out in forward and backward time by the solution to the regularized Hamiltonian $\Gamma$ with initial conditions $\gamma_0$.  If we use the initial conditions determined by $\alpha$ and $\beta$ in the previous section, then the time interval $0 \leq s \leq 12\tau$ captures the entire orbit and $\mathcal{O}(\gamma_0)$ is a closed loop in $\mathbb{R}^6$.  This orbit is \textit{Poincar\'e stable} if given any $\epsilon > 0$ there is some $\delta > 0$ so that for initial conditions $\tilde{\gamma}_0$ with $|\tilde{\gamma}_0 - \gamma_0| < \delta$, then any point on the orbit $\mathcal{O}(\tilde{\gamma}_0)$ is within $\epsilon$ of a point on the orbit $\mathcal{O}(\gamma_0)$. \\

Poincar\'e stability is generally difficult to establish in all but the simplest cases.  However, there is a necessary condition that can be computed.  Specifically, for a Hamiltonian system with Hamiltonian $\Gamma$ and a periodic orbit $\gamma(s)$ with period $T$, consider the matrix differential equation
\begin{equation}
\label{LinearStabilityEquation}
X' = JD^2\Gamma(\gamma(s)), \quad X(0) = I
\end{equation}
where $D$ denotes the derivative, $J$ is the symplectic matrix
$$J = 
\begin{bmatrix}
0 & I \\
-I & 0
\end{bmatrix}$$
with $I$ and $0$ are appropriately sized identity and zero matrices.  Then the \textit{monodromy matrix} of the orbit is the matrix $X(T)$, and the orbit is \textit{linearly stable} if the eigenvalues of $X(T)$ all have complex modulus 1 and all have multiplicity one, apart from pairs of eigenvalues equal to 1 corresponding to first integrals of the system.  \\

Linear stability can be established by considering conditions other that $X(0) = I$ as well.  Specifically, if we let $X(0) = Y_0$ be the initial condition to Equation \ref{LinearStabilityEquation}, then $Y(s) = X(s)Y_0$, so $X(T) = Y(T)Y_0^{-1}$.  Hence, $Y_0^{-1}Y(T)$ is similar to the monodromy matrix $X(T)$, and linear stability can be determined from either matrix as similarity preserves eigenvalues. \\

In the cubic setting, our choice of coordinates has already forced the integrals corresponding to center of mass, net momentum, and angular momentum to be zero.  Hence the monodromy matrix corresponding to the periodic cubic orbit should contain one pair of eigenvalues 1 corresponding to the fixed value of the Hamiltonian.  Further, it will be shown that a particular choice of $Y_0$ simplifies the calculation.\\

\subsection{Symmetries of the Orbit}
\label{subSym}

A technique by Roberts allows us to further simplify this calculation by ``factoring'' the monodromy matrix in terms of the symmetries of the orbit.  For completeness, we review the relevant results below.  Full details of the proofs can be found in \cite{bibRoberts8}. \\

By construction of the orbit as given in Lemma \ref{symmetryconstruct}, we have that
$$\gamma(s + 2\tau) = S_f\gamma(s)$$
where
\begin{equation}
\label{sfequation}
S_f =
\left[
\begin{array}{ccc|ccc}
0 & 0 & 1 & 0 & 0 & 0 \\
-1 & 0 & 0 & 0 & 0 & 0 \\
0 & 1 & 0 & 0 & 0 & 0 \\
\hline
0 & 0 & 0 & 0 & 0 & 1 \\
0 & 0 & 0 & -1 & 0 & 0 \\
0 & 0 & 0 & 0 & 1 & 0
\end{array}
\right].
\end{equation}
So $S_f$ is a time-preserving symmetry of the orbit.  This symmetry corresponds to a $120^\circ$ rotation about the line $x = y = z$, coupled with an appropriate sign change which arises in the regularized setting. \\

A time-reversing symmetry of the orbit is given by
$$\gamma(-s + 2\tau) = S_r\gamma(s)$$
where
\begin{equation}
\label{srequation}
S_r =
\left[
\begin{array}{ccc|ccc}
0 & 1 & 0 & 0 & 0 & 0 \\
1 & 0 & 0 & 0 & 0 & 0 \\
0 & 0 & 1 & 0 & 0 & 0 \\
\hline
0 & 0 & 0 & 0 & -1 & 0 \\
0 & 0 & 0 & -1 & 0 & 0 \\
0 & 0 & 0 & 0 & 0 & -1
\end{array}
\right].
\end{equation}
This can be proven using a similar technique as shown in Lemma \ref{symmetryconstruct}.  Setting $s = \tau$ gives $\gamma(\tau) = S_r\gamma(\tau)$, implying that $\gamma(\tau)$ is an eigenvector of $S_r$ with eigenvalue 1.  Directly computing this, we find that $\gamma(\tau)$ must be of the form
$$\gamma(\tau) = (a, a, b, c, -c, 0)^T$$
for suitable values of $a$, $b$, and $c$.  \\

This eigenvector has an important physical interpretation.  Halfway between the collisions on the $x = 0$ and $y = 0$, the bodies lie in the planes $x = \pm y$, with velocities orthogonal to that plane.  Travel in either orthogonal direction along the periodic orbit produces the same trajectory up to reflection across the appropriate plane.  \\

Consider the portion of the periodic orbit from $s = 0$ to $s = \tau$, or from collision in the $x = 0$ plane up to all bodies lying on the $x = \pm y$  Coupled with the time-preserving symmetry noted earlier shows that the orbit in regularized space can then be broken into 12 copies of the orbit on $0 \leq s \leq \tau$.  The remainder of the orbit is then constructed by appropriate translation, rotation, and reflection operations.  \\

\textbf{Note:} This 12-fold symmetry is similar to that of the figure-eight orbit of Moore, Chenciner, and Montgomery (see \cite{bibMoore} and \cite{bibChencinerMontgomery}), which also has a time-preserving symmetry of order 6, and a time-reversing symmetry of order 2.  \\

\subsection{Robert's Symmetry-Reduction Technique}
\label{subRoberts}

The general results in this section are presented, with proof, in Section 2 of \cite{bibRoberts8}.  The results are included here for convenience.  The application of each result to the cubic orbit is given after each statement.  Results similar to Lemmas \ref{Wform}-\ref{KEig} also appear in \cite{bibRoberts8}, but the form presented in this section is specifically applied to the cubic orbit.

\begin{lemma}
\label{Roberts21}
Suppose that $\gamma(s)$ is a $T$-periodic solution of a Hamiltonian system with Hamiltonian $\Gamma$ and time-preserving symmetry $S$ such that
\begin{enumerate}
\item There exists some $N \in \mathbb{N}$ such that $\gamma(s + T/N) = S\gamma(s)$ for all $s$,
\item $\Gamma(Sx) = \Gamma(x)$
\item $SJ = JS$
\item $S$ is orthogonal
\end{enumerate}
Then the fundamental matrix solution $X(s)$ to the linearization problem $\dot{X} = JD^2\Gamma(\gamma(s))X$ with $X(0) = I$ satisfies
$$X(s + T/N) = SX(s)S^TX(T/N).$$
\end{lemma}

We note that the matrix $S = S_f$ from Equation \ref{sfequation} satisfies all of these hypotheses with $T = 12\tau$ and $N = 6$. \\

\begin{corollary}
\label{Roberts22}
Given the hypotheses of Lemma \ref{Roberts21}, the fundamental matrix solution $X(s)$ satisfies
$$X(kT/N) = S^k(S^TX(T/N))^k$$
for any $k \in \mathbb{N}$.
\end{corollary}

In the case of the cubic orbit, this gives us that $X(12\tau) = (S_f^TX(2\tau))^6$, as $S_f^6 = I$.

\begin{lemma}
\label{Roberts24}
Suppose that $\gamma(s)$ is a $T$-periodic solution of a Hamiltonian system with Hamiltonian $\Gamma$ and time-reversing symmetry $S$ such that
\begin{enumerate}
\item There exists some $N \in \mathbb{N}$ such that $\gamma(-s + T/N) = S\gamma(s)$ for all $s$
\item $\Gamma(Sx) = \Gamma(x)$
\item $SJ = -JS$
\item $S$ is orthogonal.
\end{enumerate}
Then the fundamental matrix solution $X(s)$ to the linearization problem $\dot{X} = JD^2\Gamma(\gamma(s))X$ with $X(0) = I$ satisfies
$$X(-s + T/N) = SX(s)S^TX(T/N).$$
\end{lemma}

The matrix $S = S_r$ from Equation \ref{srequation} satisfies all of these hypotheses with $T = 12\tau$ and $N = 6$.

\begin{corollary}
\label{Roberts25}
Given the hypotheses of Lemma \ref{Roberts24}, 
$$X(T/N) = SA^{-1}S^TA, \quad A = X(T/2N).$$
\end{corollary}

In the case of the cubic orbit, noting that $S_r^T = S_r$ gives $X(2\tau) = S_rA^{-1}S_rA$ with $A = X(\tau)$.  Combining this with the earlier result, this gives us that the monodromy matrix of the cubic orbit is $X(12\tau) = (S_f^TS_rA^{-1}S_rA)^6$.  Hence, we can evaluate the stability of the orbit by evaluating the relevant differential equations along only a twelfth of the orbit. \\

Roberts also gives similar results for the case where the initial conditions given in Equation \ref{LinearStabilityEquation} are not the identity matrix.  These are listed below.

\begin{corollary}
\label{RobertsRemark1}
If $Y(s)$ is the fundamental matrix solution with $X(0) = Y_0$, then
$$Y(s + T/N) = SY(s)Y_0^{-1}S^TY(T/N),$$
and so
$$Y(kT/N) = S^kY_0(Y_0^{-1}S^TY(T/N))^k$$
\end{corollary}

\begin{corollary}
\label{RobertsRemark2}
If $Y(s)$ is the fundamental matrix solution with $X(0) = Y_0$, then
$$Y(-s + T/n) = SY(s)Y_0^{-1}S^TY(T/N),$$
and so
$$Y(T/N) = SY_0B^{-1}S^TB, \quad B = Y(T/2N)$$
\end{corollary}

Combining these with previous results gives that for an arbitrary $X(0) = Y_0$, the resulting matrix solution $Y(s)$ to Equation \ref{LinearStabilityEquation} satisfies
$$Y(12\tau) = Y_0(Y_0^{-1}S_f^TS_rY_0B^{-1}S_rB)^6$$
so
$$X(12\tau) = Y_0(Y_0^{-1}S_f^TS_rY_0B^{-1}S_rB)^6Y_0^{-1},$$
where $B = Y(\tau)$. \\

Define $W = Y_0^{-1}S_f^TS_rY_0B^{-1}S_rB$.  Then $X(12\tau) = Y_0W^6Y_0^{-1}$, and stability of the cubic orbit is thus reduced to determining the eigenvalues of $W$. \\

For a properly chosen initial condition matrix $Y_0$, some additional simplification of the calculation can be done.  Again from \cite{bibRoberts8},
\begin{lemma}
Suppose that $W$ is a symplectic matrix satisfying
$$\frac{1}{2}(W + W^{-1}) =
\begin{bmatrix}
K & 0 \\
0 & K^T
\end{bmatrix}.$$
Then $W$ has all eigenvalues on the unit circle if and only if the eigenvalues of $K$ lie in the real interval $[-1, 1]$.
\label{Wform}
\end{lemma}
Proper choice of the matrix $Y_0$ will give $W$ of the required form.

\begin{lemma}
Setting $\delta = \sqrt{2}/2$ and 
\begin{equation}
Y_0 = 
\left[
\begin{array}{ccc|ccc}
1 & 0 & 0 & 0 & 0 & 0 \\
0 & 0 & -\delta & 0 & \delta & 0 \\
0 & 0 & \delta & 0 & \delta & 0 \\
\hline
0 & 0 & 0 & 1 & 0 & 0 \\
0 & -\delta & 0 & 0 & 0 & -\delta \\
0 & -\delta & 0 & 0 & 0 & \delta
\end{array}
\right]
\label{matrixy0}
\end{equation}
gives a matrix $W$ of the form in Lemma \ref{Wform}.
\end{lemma}

\begin{proof}
Let
\begin{equation*}
\Lambda =
\begin{bmatrix}
I & 0 \\
0 & -I \\
\end{bmatrix}
\end{equation*}
where $I$ and $0$ represent $3 \times 3$ identity and zero matrices, respectively.  Then direct calculation yields $-Y_0^{-1}S_f^TS_rY_0 = \Lambda$. \\

Set $D = -B^{-1}S_rB$.  Then by definition of $W$ we have that $W = \Lambda D$.  Since $D^2 = \Lambda^2 = I$, then we know that $W^{-1} = D\Lambda$.  Since $B$ is symplectic, writing
\begin{equation*}
B =
\begin{bmatrix}
B_1 & B_2 \\
B_3 & B_4
\end{bmatrix}
\quad \text{ and } \quad
S_r =
\begin{bmatrix}
S & 0 \\
0 & -S
\end{bmatrix}
\end{equation*}
then the formula for the inverse of a symplectic matrix gives
\begin{equation*}
B^{-1} =
\begin{bmatrix}
B_4^T & -B_2^T \\
-B_3^T & B_1^T
\end{bmatrix}.
\end{equation*}
Directly computing $D$ gives

\begin{equation*}
D =
\begin{bmatrix}
K^T & L_1 \\
-L_2 & -K
\end{bmatrix}
\end{equation*}
with $K$, $L_1$, and $L_2$ defined up to sign by matrix multiplication.  Then
\begin{equation*}
W = \Lambda D =
\begin{bmatrix}
K^T & L_1 \\
L_2 & K
\end{bmatrix}
\quad \text{ and } \quad
W^{-1} = D \Lambda =
\begin{bmatrix}
K^T & -L_1 \\
-L_2 & K
\end{bmatrix}
\end{equation*}
and
$$\frac{1}{2}(W + W^{-1}) =
\begin{bmatrix}
K^T & 0 \\
0 & K
\end{bmatrix}$$
as required.
\end{proof}

As noted earlier, our coordinate system has already made use of the first integrals corresponding to center of mass, net momentum, and angular momentum in this setting.  There is an additional pair of eigenvalues $1$ in the monodromy matrix corresponding to the remaining first integral, the Hamiltonian itself.  These can be found, with eigenvector, as shown below. \\

\begin{lemma}
The matrix $K^T$ has a right eigenvector $[1 \ 0 \ 0]^T$ with corresponding  eigenvalue $1$.
\label{KEig}
\end{lemma}
\begin{proof}
Let $v = Y_0^{-1}\gamma'(0)/||\gamma'(0)|| = Y_0^{T}\gamma'(0)/||\gamma'(0)||$.  Since $Y_0$ is orthogonal and $S_r$ is symmetric, we have
$$W = Y_0^{-1}S_f^TS_rY_0B^{-1}S_rB = Y_0^{T}S_f^TS_rY_0B^{-1}S_r^TB = Y_0^TS_f^TY(2\tau)$$
by Corollary \ref{RobertsRemark2} with $s = 0$. \\

Define $\gamma(s)$ to be the periodic orbit with initial conditions defined in Section \ref{secPeriodicOrbit}.  Since $\gamma'(s)$ is a solution to $\dot{\xi} = JD^2\Gamma(\gamma(s))\xi$ and $$\gamma'(0) = Y(0)Y_0^{-1}\gamma'(0) = Y(0)v,$$
then
$$\gamma'(s) = Y(s)Y_0^{-1}\gamma'(0) = Y(s)v.$$
This implies
$$Y_0^{-1}S_f^T\gamma'(2\tau) = Y_0^TS_f^TY(2\tau)v = Wv.$$
Since 
$$\gamma'(0) = (2\sqrt{2}\alpha^4, 0, 0, 0, 0, 0)$$
and
$$\gamma'(2\tau) = (0, -2\sqrt{2}\alpha^4, 0, 0, 0, 0)$$
with $\alpha$ as defined in Equation \ref{alphabeta}, we have
$$S_f^T\gamma'(2\tau) = \gamma'(0).$$
Then $$Wv = Y_0^{-1}S_f^T\gamma'(2\tau) = Y_0^TS_f^TS_f\gamma'(0) = Y_0^TY_0v = v.$$
So $v$ is an eigenvector of $W$ with eigenvalue $1$. \\

Since $\gamma'(0)$ is known, we have that $v = Y_0^{-1}e_1$, where
$$e_1 = [1 \ 0 \ 0 \ 0 \ 0 \ 0]^T.$$
Direct calculation gives that $v = e_1$.  Then, since $W$ satisfies the relation given in Lemma $\ref{Wform}$, $K^T$ must have eigenvector $[1 \ 0 \ 0]^T$ with eigenvalue $1$ as claimed.
\end{proof}

As a consequence, we know that the matrix $K$ must be of the form
\begin{equation}
K =
\begin{bmatrix}
1 & 0 & 0 \\
* & k_{22} & k_{23} \\
* & k_{32} & k_{33}
\end{bmatrix}
\end{equation}
and so the eigenvalues of the lower-right $2 \times 2$ block will determine stability.

\section{Stability Results}
\label{secStabRes}

Using the matrix $Y_0$ from Equation \ref{matrixy0}, we find the matrix $B = Y(\tau)$ numerically with the initial conditions from Equation \ref{alphabeta}. Then the matrix $K$ is given numerically by
\begin{equation*}
K = 
\begin{bmatrix}
1.0007 & 0.0004 & -0.0001 \\
-0.9038 & 0.3487 & 0.1926 \\
1.7654 & -1.1211 & -1.1241
\end{bmatrix}
\end{equation*}
The values given for the $k_{12}$ and $k_{13}$ entries are the result of propagation of numerical error in the calculation.  Assuming they are zero as proven earlier, 
the eigenvalues from the lower-right $2 \times 2$ block of $K$ are given by a simple application of the quadratic formula.  We find
$$\lambda_1 = 0.1836, \quad \lambda_2 = -0.95899$$
As a consequence of Lemma \ref{Wform}, we have the following
\begin{theorem}
The cubic orbit described throughout this paper is linearly stable.
\end{theorem}

We seek to give evidence of higher-order stability of the cubic orbit.  Using $E = -1$ and the values of $\alpha$ and $\beta$ from Equation \ref{alphabeta}, we set
\begin{align*}
\gamma_0 = &(0, \alpha + r\cos(a)\cos(b), \alpha + r\cos(a)\sin(b), ... \\
&\sqrt{2}, \beta + r\sin(a)\cos(c), -\beta + r\sin(a)\sin(c))
\end{align*}
where
\begin{align*}
a, b, c &\in \{0, \pi/6, \pi/3, \pi/2, \ldots, 11\pi/6\} \\
r &\in \{0.005, 0.010, 0.015, \ldots, 0.100\}.
\end{align*}
The equations of motion are run up to 200 collisions at $Q_1 = 0$ for each possible combination of $a$, $b$, $c$, and $r$.  Integration is preemptively terminated after a time length of $s = 1$ has occurred since the last $Q_1 = 0$ collision.  This time cutoff value seems reasonable given the length of the period $12\tau = 0.124736$.  We track the distance from $\pm \gamma(0)$ at those collision times.  For all values of $r$ for which all 200 collisions were achieved on all values of $a$, $b$, and $c$ tested, the maximum distance from $\pm \gamma(0)$ at collision is given in the table below.
\begin{equation*}
\begin{array}{|c|c|}
\hline
r & \text{dist}_{max} \\
\hline
0.005 & 0.0391 \\
\hline
0.010 & 0.0824 \\
\hline
0.015 & 0.1289 \\
\hline
0.020 & 0.1833 \\
\hline
0.025 & 0.2574 \\
\hline
\end{array}
\end{equation*}
For all values of $r \geq 0.030$, there is at least one value of $a$, $b$, and $c$ for which fewer than 200 $Q_1 = 0$ instances occur.  For example, when $r = 0.030$ and $a = b = \pi/6$, $c = 11\pi/6$, only 34 instances of $Q_1 = 0$ are recorded before the integration is terminated, giving evidence of instability at this distance from the periodic orbit. \\

\section{Acknowledgements}

This section will be completed later.

\bibliographystyle{plain}

\bibliography{CubicNotes}

\begin{thebibliography}{10}

\bibitem{bibBS1}
Lennard Bakker and Skyler Simmons.
\newblock Stability of the rhomboidal symmetric-mass orbit.
\newblock {\em Disc. Cont. Dyn. Sys. A}, 35(1):1--23, 2015.

\bibitem{bibBMS}
Lennard~F. Bakker, Scott Mancuso, and Skyler~C. Simmons.
\newblock Linear stability for some symmetric periodic simultaneous binary
  collision orbits in the planar pairwise symmetric four-body problem.
\newblock {\em J. Math. Anal. Appl.}, 392(2):136--147, 2012.

\bibitem{bibBOYS}
Lennard~F. Bakker, Tiancheng Ouyang, Duokui Yan, and Skyler Simmons.
\newblock Existence and stability of symmetric periodic simultaneous binary
  collision orbits in the planar pairwise symmetric four-body problem.
\newblock {\em Celestial Mech. Dynam. Astronom.}, 110(3):271--290, 2011.

\bibitem{bibBOYSR}
Lennard~F. Bakker, Tiancheng Ouyang, Duokui Yan, Skyler Simmons, and Gareth~E.
  Roberts.
\newblock Linear stability for some symmetric periodic simultaneous binary
  collision orbits in the four-body problem.
\newblock {\em Celestial Mech. Dynam. Astronom.}, 108(2):147--164, 2010.

\bibitem{bibBV}
L\'{u}cia de~Fatima Brand\~{a}o and Claudio Vidal.
\newblock Periodic solutions of the elliptic isosceles restricted three-body
  problem with collision.
\newblock {\em J. Dynam. Differential Equations}, 20(2):377--423, 2008.

\bibitem{bibBDV}
L\'{u}cia Brand\~{a}o Dias, Joaqu\'{\i}n Delgado, and Claudio Vidal.
\newblock Dynamics and chaos in the elliptic isosceles restricted three-body
  problem with collision.
\newblock {\em J. Dynam. Differential Equations}, 29(1):259--288, 2017.

\bibitem{bibBroucke}
Roger Broucke.
\newblock On the isosceles triangle configuration in the planar general three
  body problem.
\newblock {\em Astronomy \& Astrophysics}, 73(3):303--313, 1979.

\bibitem{bibChencinerMontgomery}
Alain Chenciner and Richard Montgomery.
\newblock A remarkable periodic solution of the three-body problem in the case
  of equal masses.
\newblock {\em Ann. of Math. (2)}, 152(3):881--901, 2000.

\bibitem{bibGPSV}
Marcel Guardia, Jaime Paradela, Tere~M. Seara, and Claudio Vidal.
\newblock Symbolic dynamics in the restricted elliptic isosceles three body
  problem.
\newblock {\em J. Differential Equations}, 294:143--177, 2021.

\bibitem{bibHE}
M.~H\'enon.
\newblock Stability of interplay orbits.
\newblock {\em Cel. Mech.}, 15:243--261, 1977.

\bibitem{bibHM}
Jarmo Hietarinta and Seppo Mikkola.
\newblock Chaos in the one-dimensional gravitational three-body problem.
\newblock {\em Chaos}, 3(2):183--203, 1993.

\bibitem{bibHuang}
Hsin-Yuan Huang.
\newblock Schubart-like orbits in the {N}ewtonian collinear four-body problem:
  a variational proof.
\newblock {\em Discrete Contin. Dyn. Syst.}, 32(5):1763--1774, 2012.

\bibitem{bibKTXY}
Wentian Kuang, Tiancheng Ouyang, Zhifu Xie, and Duokui Yan.
\newblock The {B}roucke-{H}\'{e}non orbit and the {S}chubart orbit in the
  planar three-body problem with two equal masses.
\newblock {\em Nonlinearity}, 32(12):4639--4664, 2019.

\bibitem{bibLeviCivita}
T.~Levi-Civita.
\newblock Sur la r\'{e}gularisation du probl\`eme des trois corps.
\newblock {\em Acta Math.}, 42(1):99--144, 1920.

\bibitem{bibMartinez}
Regina Mart{\'{\i}}nez.
\newblock On the existence of doubly symmetric ``{S}chubart-like'' periodic
  orbits.
\newblock {\em Discrete Contin. Dyn. Syst. Ser. B}, 17(3):943--975, 2012.

\bibitem{bibMcGehee1}
Richard McGehee.
\newblock A stable manifold theorem for degenerate fixed points with
  applications to celestial mechanics.
\newblock {\em J. Differential Equations}, 14:70--88, 1973.

\bibitem{bibMISC}
R.~Moeckel.
\newblock Heteroclinic phenomena in the isosceles three-body problem.
\newblock {\em SIAM J. Math. Anal.}, 15(5):857--876, 1984.

\bibitem{bibMoeckel}
Richard Moeckel.
\newblock A topological existence proof for the {S}chubart orbits in the
  collinear three-body problem.
\newblock {\em Discrete Contin. Dyn. Syst. Ser. B}, 10(2-3):609--620, 2008.

\bibitem{bibMoore}
Cristopher Moore.
\newblock Braids in classical dynamics.
\newblock {\em Phys. Rev. Lett.}, 70(24):3675--3679, 1993.

\bibitem{bibNewton}
I.~S. Newton.
\newblock {\em Philosophiae naturalis principia mathematica}.
\newblock William Dawson \& Sons, Ltd., London, undated.

\bibitem{bibOY}
Tiancheng Ouyang and Duokui Yan.
\newblock Simultaneous binary collisions in the equal-mass collinear four-body
  problem.
\newblock {\em Electron. J. Differential Equations}, pages No. 80, 34, 2015.

\bibitem{bibOY3}
Tiancheng Ouyang, Duokui Yan, and Skyler Simmons.
\newblock Periodic solutions with singularities in two dimensions in the
  $n$-body problem.
\newblock {\em Rocky Mtn. J. Math.}, 42(4):1601--1614, 2012.

\bibitem{bibRoberts8}
Gareth~E. Roberts.
\newblock Linear stability analysis of the figure-eight orbit in the three-body
  problem.
\newblock {\em Ergodic Theory Dynam. Systems}, 27(6):1947--1963, 2007.

\bibitem{bibRoySteves}
Archie~E. Roy and Bonnie~A. Steves.
\newblock The {C}aledonian symmetrical double binary four-body problem. {I}.
  {S}urfaces of zero-velocity using the energy integral.
\newblock {\em Celestial Mech. Dynam. Astronom.}, 78(1-4):299--318 (2001),
  2000.

\bibitem{bibST1}
M.~Saito and K.~Tanikawa.
\newblock The rectilinear three-body problem using symbol sequence i. role of
  triple collisions.
\newblock {\em Celest. Mech. Dynam. Astron.}, 98:95--120, 2007.

\bibitem{bibST2}
M.~Saito and K.~Tanikawa.
\newblock The rectilinear three-body problem using symbol sequence ii: Role of
  periodic orbits.
\newblock {\em Celest. Mech. Dynam. Astron.}, 103:191--207, 2009.

\bibitem{bibST3}
M.~Saito and K.~Tanikawa.
\newblock Non-schubart periodic orbits in the rectilinear three-body problem.
\newblock {\em Celest. Mech. Dynam. Astron.}, 107:397--407, 2010.

\bibitem{bibJS}
J.~Schubart.
\newblock Numerische {A}ufsuchung periodischer {L}\"osungen im
  {D}reik\"orperproblem.
\newblock {\em Astr. Nachr.}, 283:17--22, 1956.

\bibitem{bibShib}
Mitsuru Shibayama.
\newblock Minimizing periodic orbits with regularizable collisions in the
  {$n$}-body problem.
\newblock {\em Arch. Ration. Mech. Anal.}, 199(3):821--841, 2011.

\bibitem{bibSim}
Skyler Simmons.
\newblock Stability of {B}roucke's isosceles orbit.
\newblock {\em Discrete Contin. Dyn. Syst.}, 41(8):3759--3779, 2021.

\bibitem{bibSSS}
A.~Sivasankaran, B.~Steves, and W.~Sweatman.
\newblock A global regularisation for integrating the caledonian symmetric
  four-body problem.
\newblock {\em Celest. Mech. Dynam. Astron.}, 107:157--168, 2010.

\bibitem{bibVE}
Andrea Venturelli.
\newblock A variational proof of the existence of von {S}chubart's orbit.
\newblock {\em Discrete Contin. Dyn. Syst. Ser. B}, 10(2-3):699--717, 2008.

\bibitem{bibYan1}
Duokui Yan.
\newblock Existence and linear stability of the rhomboidal periodic orbit in
  the planar equal mass four-body problem.
\newblock {\em J. Math. Anal. Appl.}, 388(2):942--951, 2012.

\bibitem{bibYan2}
Duokui Yan.
\newblock Existence of the {B}roucke periodic orbit and its linear stability.
\newblock {\em J. Math. Anal. Appl.}, 389(1):656--664, 2012.

\bibitem{bibYan3}
Duokui Yan.
\newblock A simple existence proof of {S}chubart periodic orbit with arbitrary
  masses.
\newblock {\em Front. Math. China}, 7(1):145--160, 2012.

\end{thebibliography}

\end{document}